\documentclass{article}
\usepackage{geometry}
\usepackage{hyperref}
\usepackage{amssymb}
\usepackage{mathrsfs}
\usepackage{verbatim}
\usepackage{extarrows}
\usepackage{latexsym, cite, bm, amsthm, amsmath}
\usepackage{epsfig,graphicx, epsf,wrapfig}
\usepackage{epstopdf}
\usepackage{amsmath}\numberwithin{equation}{section}
\graphicspath{{figure/}}
\usepackage{subfigure}
\usepackage{tikz}
\usepackage{tabularx,multirow}
\usepackage{titletoc}
\usepackage{stmaryrd}
\usepackage{setspace} %

\newcommand{\dif}{\mathop{}\mathrm{d}}
\newcommand{\n}{\mathop{}\mathrm{n}}
\newcommand{\bbR}{\mathbb{R}}

\newcommand{\bl}{\boldsymbol}
\newcommand{\T}{\mathcal{T}}
\newcommand{\E}{\mathcal{E}}

\newcommand{\wt}{\widetilde}
\newcommand{\lal}{\langle}
\newcommand{\ral}{\rangle}
\newcommand{\pt}{\partial}

\newtheorem{thm}{Theorem}[section]

\newtheorem{lem}{Lemma}[section]
\newtheorem{cro}{Corollary}[section]
\newtheorem{re}{Remark}[section]
\newcommand{\wuhao}{\fontsize{10.5pt}{12.6pt}\selectfont}        

\title{High-order primal mixed finite element method for boundary-value correction on curved domain}
\date{}
\author{Yongli Hou,
Yi Liu\footnote{Corresponding author}~~and Tengjin Zhao
}

\begin{document}
\maketitle
\setstretch{1.0}                            
\setlength{\parindent}{2em}                 
\setlength{\parskip}{3pt plus1pt minus1pt}  

\setlength{\floatsep}{10pt plus 3pt minus 2pt}      
\setlength{\abovecaptionskip}{2pt plus 1pt minus 1pt} 
\setlength{\belowcaptionskip}{3pt plus1pt minus2pt} 

\renewcommand\figurename{Fig.}
\renewcommand\tablename{Tab.}

\wuhao
\begin{abstract}
  This paper addresses the non-homogeneous Neumann boundary condition on domains 
  with curved boundaries.
  We consider the Raviart-Thomas element $(RT_k)$ of degree $k\geq 1$ on triangular mesh.
  on a triangular mesh. A key feature of our boundary value correction method is 
  the shift from the true boundary to a surrogate boundary. 
  We present a high-order version of the method, achieving an $O(h^{k+1/2})$ convergence 
  in $L^2$-norm estimate for the velocity field 
  and an $O(h^k)$ convergence in $H^1$-norm estimate for the pressure.
  Finally, numerical experiments validate our theoretical results.

\end{abstract}

\vspace{0.9em}
\section{\textbf Introduction}
Many practical problems arising in science and engineering often involve domains with curved boundaries.
For the domain $\Omega$ with curved boundaries, 
the geometric error between {the} curved boundary $\Gamma$
and the approximating boundary $\Gamma_h$ leads to
a loss of accuracy for high-order elements 
\cite{Strang_Berger_The_change_in_solution_due_to_change_in_domain1973,Thomee_Polygonal_domain_approximation_in_Dirichlet's_problem1973}.
There are two main strategies to address this issue.
Both the isoparametric finite element method 
\cite{isoparametric_1968,Optimal_isoparametric_finite_elements_1986} and 
the isogeometric analysis 
\cite{Isogeometric_analysis_Hughes_2005,Isogeometric_Analysis_Cottrell_2009} aim to
reduce the geometric error without modifying the variational form. 
Another strategy is the boundary value correction method \cite{1972}, 
which directly solves on a polygonal approximation domain $\Omega_h$,
and focuses on a modified variational formulation. 
Recently, there has been a growing body of research on boundary value correction, 
{including studies on} the discontinuous Galerkin method \cite{Cockburn_Boundary_conforming_DG_2009,Cockburn2012}, 
the shifted boundary method \cite{Part1,high_order_SBM},
the cut finite element method (cutFEM) with boundary value correction \cite{Burman_boundary_value_correction_2018}
and the weak Galerkin method with boundary value corrected 
\cite{Yiliu_WG}({ I will be put that bvc paper on axriv}), among others.

The treatment of curved boundaries for the mixed Poisson problems with Neumann boundary conditions
is {seldom discussed} in the literature.
In \cite{Bertrand_2014_lowest,Bertrand_2014_High_order,Bertrand_2016_High_order},
the boundary condition on $\Gamma_h$ is imposed {in} the discrete space.
In \cite{Bertrand_2014_lowest}, authors study the least squares method for $RT_0$
and verify a loss of accuracy for  $RT_1$ through numerical experiments, but no theoretical analysis is provided.
{Later}, they {extended the analysis} to high-order Raviart-Thomas element 
in \cite{Bertrand_2014_High_order}
and parametric Raviart-Thomas elements {in} \cite{Bertrand_2016_High_order}. 
Recently,
the cutFEM {for curved boundaries applied to} $RT_k,k\geq 1$ {has been studied}
in \cite{Puppi_cutFEM}.
In their work, the boundary conditions are weakly enforced {within} the weak formulation.
This work {achieves} an $O(h^k)$ convergence {rate in the} $L^2$-norm estimate 
for the velocity {field}, 
which is {considered} a suboptimal convergence order.

The purpose of this paper is to provide a convergence analysis of the
primal mixed element approximation \cite{book:FE_RT} 
using the pairing $L^2(\Omega)$ and $H^1(\Omega)$ 
for mixed Poisson problems with curved boundary.
In the context of the mixed finite element method, Neumann boundary conditions are essential.
Following the approach in \cite{fractual_non_matching_2012,Puppi_mixed}, 
this paper weakly imposes the Neumann boundary conditions in the weak formulation.
Moreover, this paper adopts the boundary value correction method, 
which avoids cut elements and reduces complexity of implementation.
The method employs Taylor expansion to address the discrepancy in normal flux between
 $\Gamma$ and $\Gamma_h$.
 Furthermore, this work includes a 
 rigorous analysis of the loss of approximation accuracy for high-order elements.
 Finally, in comparison to the cutFEM on curved boundaries discussed in \cite{Puppi_cutFEM}, 
 this paper achieves a better convergence order of $O(h^{k+1/2})$ for the velocity field. 

   This paper is organized as follows. In section \ref{sec2:notation}, 
   we introduce some notations and preliminaries;
   In section \ref{sec3:model_problem}, 
   we describe the model problem and introduce the boundary value correction method;
   In section \ref{sec4:discrete}, 
   {we establish the discrete space and variational {form} and analyze {its} well-posedness};
   In section \ref{sec5:error}, 
   the energy error estimate and the $L^2$ error estimate are proved;
   In section \ref{sec6:without_correction}, 
   we analysis the problem without {the boundary value correction};
   In section \ref{sec7:numerical}, 
   we {present} several numerical experiments to verify the theoretical results;
   { we conclude in \ref{sec8:conclusion} with our findings.}

\section{ Notations and preliminaries}\label{sec2:notation}
Throughout {this} paper, let $\Omega$ be a connected open set in $\bbR^2$ with Lipschitz continuous boundary $\Gamma$.
We assume that $\Omega$ is approximated by a polygonal domain $\Omega_h$ and denote by $\Gamma_h$ the boundary of $\Omega_h$.
Let $\T_h$ denote a family of triangular meshes {for} $\Omega_h$.
We require that all the vertices of $\T_h$ lying on $\Gamma_h$ also lie on $\Gamma$,
{ensuring} $\T_h$ is a body-fitted triangular partition of $\Omega$.
For each $K\in{\T}_h$, let $h_K=\text{diam(K)}$ and $h=\max_{K\in{\T}_h}h_K$.
We assume the mesh is shape-regular; that is, there exists a constant $\sigma>0$, 
independent {of} $h$,
such that $ \max_{K\in{\T}_h}\frac{h_K}{\rho_K}\leq \sigma$, 
{where $\rho_K$ is the diameter of the largest ball inscribed in $K$.}
{Furthermore, we assume that the mesh is quasi-uniform; that is there exists $\tau>0$, 
independent of $h$,
such that $\min_{K\in{\T}_h} h_K\geq \tau h$.}
{Let $\E_h$ denote the set of all edges in $\T_h$.
Define $\E_h^o$ as the set of all interior edges and $\E_h^b=\E_h\setminus{\E_h^o}$.}
Denote by $\T_h^b$ all mesh elements containing at least one edge in $\E_h^b$ 
and $\T_h^o=\T_h\setminus\T_h^b$.
Let $e$ be an interior edge shared by two elements $K_1$ and $K_2$, 
we define the jump $[\cdot]$ on $e$ for scalar functions $q$ as follows:
$$[q]=q|_{K_1}-q|_{K_2}.$$

We {adopt} standard definitions for the Sobolev spaces {as presented in} \cite{FEM_Brenner}.
Let $H^m(S)$, for $m\in\bbR$ and $S\subset\bbR^2$ be the usual Sobolev space with
associated norm $\|\cdot\|_{m,S}$ and seminorm $|\cdot|_{m,S}$.
{When $m=0$, the space $H^0(S)$ coincides with the square integrable space $L^2(S)$.}
{We define 
$$H^m(\T_h):=\prod_{K\in\T_h}H^m(K)$$ 
with seminorm 
$$|\cdot|_{m,\T_h}:=(\sum_{K\in\T_h}|\cdot|_{m,K}^2)^{1/2}.$$
We denote $\bl H^m(S)$ and $\bl H^m(\T_h)$ representing the corresponding vector space.}

We denote $L_0^2(S)$ as the mean value free subspace of $L^2(S)$.
For $m\geq 0$, the above notation extends to a portion 
$s\subset\Gamma$ or $s\subset\Gamma_h$.
For example, 
let $\|\cdot\|_{m,s}$ be the {Sobolev} norm on $s$.
Denote by $(\cdot,\cdot)_S$ the $L^2$ inner-product on $S\subset\bbR^2$,
and by $\lal\cdot,\cdot\ral_s$ the duality pair on {$s\subset\Gamma$ or $s\subset\Gamma_h$}.
Finally, all the above-defined notations can be easily extended to vector spaces using the standard product.
{Moreover,
\begin{equation*}
\begin{aligned}
H(\mathrm{div},S)&=\{\bl v\in \bl L^2(S):\mathrm{div}\,\bl v \in L^2(S)\},
\end{aligned}
\end{equation*}
and its subspace $H_g(\mathrm{div},S)$  for the given $g\in H^{-\frac12}(\partial S)$ is defined as
\begin{equation*}
\begin{aligned}
H_g(\mathrm{div},S)&=\{\bl v\in H(\mathrm{div},S): \bl v\cdot\bl\n=g \quad \text{on}~\partial S \}.
\end{aligned}
\end{equation*}}

The norm in $H(\mathrm{div},S)$ is defined by
$$\| \bl v\|_{H(\mathrm{div},S)}=(\| \bl v\|_{0,S}^2+\|\mathrm{div}\, \bl v\|_{0,S}^2)^{1/2}.$$

We will also use the notation
$$M_k(\bl v)=\max_{|\alpha|\leq k}\sup_{x\in\Omega_h}|D^\alpha\bl v(x)|.$$

In the remainder of the paper, we use the notation $\lesssim$ to denote {less than} or equal
 to a constant
and the analogous notation $\gtrsim$ to denote {greater than} or equal to a constant.

Hereafter, we collect some well-known inequalities that are used in {this} paper.
\begin{lem}\label{lem:TRACE}
(Trace Inequality \cite{FEM_Brenner}). 
For any $K\in\T_h$ and $v\in H^1(K)$, we have
\begin{equation*}
\begin{aligned}
\|v\|_{0,\pt K}\lesssim h_K^{-1/2}\| v\|_{0,K}+h_K^{1/2}|v|_{1,K}.
\end{aligned}
\end{equation*}
\end{lem}
\begin{lem}\label{lem:Poincare_inequality}
(Poncar\'{e}-Friedrichs inequality \cite{Poincare_Friedrichs_2003}). For any $v\in H^1(\T_h)$, one gets
\begin{equation*}
\begin{aligned}
\|v\|_{0,\Omega_h}^2\lesssim \sum_{K\in \T_h}|v|_{1,K}^2+\sum_{e\in\E_h^o}h^{-2}\left(\int_e [v]\dif s\right)^2+\left(\int_{\Omega_h} v\dif x\right)^2.
\end{aligned}
\end{equation*}
\end{lem}
\begin{lem}\label{lem:inverse}
(Inverse Inequality \cite{FEM_Brenner}).  For any $K\in\T_h$ and $q\in P_l(K), \,0\leq m\leq l$, we have
\begin{equation*}
\begin{aligned}
|q|_{l, K}\lesssim h_K^{m-l}| q|_{m,K}.
\end{aligned}
\end{equation*}
\end{lem}

\section{ Model problem and the boundary value correction method}\label{sec3:model_problem}
In this section, we briefly introduce the model problem and the boundary value correction method \cite{1972}.
Given $f\in L^2(\Omega), g_N\in H^{-1/2}(\Gamma)$,
there exists a unique solution pair $(\bl u,p)\in {H_{g_N}(\mathrm{div},\Omega)}\times L_0^2(\Omega)$ such that
 \begin{equation}
\begin{aligned}
\bl u=&~-\nabla p ,\quad &\text{in} \,\Omega ,&\\ \label{primal_problem}
        \mathrm{div}\, \bl u=&~f ,\quad &\text{in} \, \Omega,\\
         \bl u\cdot\bl\n=&~g_N ,\quad &\text{on} \, \Gamma,
\end{aligned}
\end{equation}
where $\bl{\n}$ denotes the unit outward normal on $\Gamma$.
{Problem (\ref{primal_problem}) is well-posed as long as the following compatibility condition holds:
\begin{align}\label{g_comp_cond}
    \int_\Omega f\dif x=\int_\Gamma g_N\dif s.
\end{align}
}
{To} shift the boundary date from $\Gamma$ to $\Gamma_h$,
we assume that there exists a map $M_h:\Gamma_h\rightarrow\Gamma$ 
{defined as follows}
\begin{equation}
\begin{aligned}\label{map}
 M_h(\bl x_h):=\bl x_h+\delta_h(\bl x_h)\bl \nu_h(\bl x_h),
 \end{aligned}
\end{equation}
as shown in Fig. \ref{figure1}(b), where $\bl\nu_h$ is {an unit} 
vector defined on $\Gamma_h$, and $\delta_h(\bl x_h)=|M_h(\bl x_h)-\bl x_h|$.
Denote by $\bl x:=M_h(\bl x_h)$ and by $\wt{\bl\n}:=\bl\n\circ M_h$. 
{As shown in} \cite{Burman_boundary_value_correction_2018}, we have
\begin{equation}
\begin{aligned}\label{deta}
\delta=\sup_{\bl x_h\in\Gamma_h}\delta_h(\bl x_h)\lesssim h^2,\qquad \|\wt{\bl\n}-\bl\n_h\|_{L^\infty(\Gamma_h)}\lesssim h,
\end{aligned}
\end{equation}
where $\bl{\n}_h$ denotes the unit outward normal on $\Gamma_h$.

\begin{re}
In this paper, we do not specify the map $M_h$. We only require that the distance function $\delta_h(\bl x_h)$ satisfies (\ref{deta}).
In existing literatures, { $M_h$ is usually defined as the closest point map}, 
which is more complicated to implement.
\end{re}
\begin{figure}[!htbp]
\centering
\subfigure[]
{
  \includegraphics[height=3cm,width=5cm]{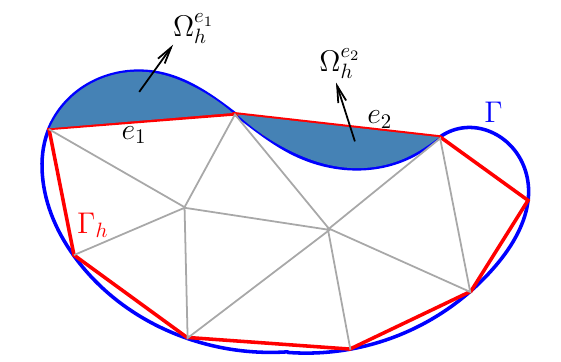}
  }\qquad\qquad
\subfigure[]
  {
\begin{tikzpicture}
\draw[red,very thick] (-1.71,0.8)--(1.69,0.8);
\draw[blue, thick] (1.8,0.6) arc (25:155:2);
\draw[gray, thick,->] (-0.3,0.8)--(-0.6,1.705);
\draw[blue, thick,->](-0.6,1.705)--(-0.85,2.5);
\draw (-0.1,0.7) node [below] {${\bm x_h}$};
\draw[red] (1,0.4) node [right] {$\Gamma_h$};
\draw[blue] (1.8,0.6) node[right] {$\Gamma$};
\draw (0.15,1.2) node {$\delta_h(\bm x_h)$};
\draw (-0.7,1.6) node[below] {$\bm x$};
\draw (-0.2,2.2) node[left] {$\bl \nu_h$};
\filldraw [black](-0.3,0.8) circle [radius=2pt];
\filldraw [black](-0.6,1.67) circle [radius=2pt];
\end{tikzpicture}
}\quad
\caption{(a). The true boundary $\Gamma$ (blue curve), the approximated boundary $\Gamma_h$ (red lines) and {the typical region $\cup_{e\in\mathcal{E}_{h}^b}\Omega_h^{e}$ bounded by $\Gamma$ and $\Gamma_h$.} (b). The distance $\delta_h(\bl x_h)$ and the unit vector $\bl{\nu_h}$ to $\Gamma_h$. }
\label{figure1}
\end{figure}
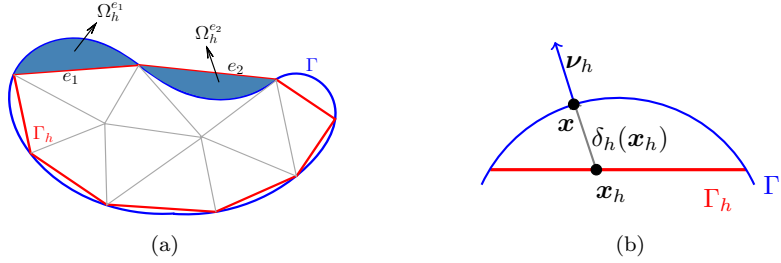

Assuming  $\bl v$ is sufficiently smooth in the strip between $\Gamma$ and $\Gamma_h$ 
{to admit} an $m$-th order Taylor expansion pointwise
$$\bl v(M_h(\bl x_h))=\sum_{j=0}^{m}\frac{\delta_h^j(\bl x_h)}{j!}
\partial_{\bl\nu_h}^j \bl v(\bl x_h)+R^m\bl v(\bl x_h),\qquad \text{on} \, \Gamma_h,$$
where $\partial_{\bl\nu_h}^j$ is the $j$-th partial derivative 
in the $\bl\nu_h$ direction, and the remainder $R^m\bl v(\bl x_h)$ satisfies
$$|R^m\bl v(\bl x_h)|=o(\delta^m).$$

{For notational convenience, we define}
\begin{equation}
\begin{aligned}\label{T1}
T^m\bl v:=\sum_{j=0}^{m}\frac{\delta_h^j(\bl x_h)}{j!}\partial_{\bl\nu_h}^j \bl v(\bl x_h), \quad
T_1^m\bl v:=\sum_{j=1}^{m}\frac{\delta_h^j(\bl x_h)}{j!}\partial_{\bl\nu_h}^j \bl v(\bl x_h).
\end{aligned}
\end{equation}
Both $T^m \bl v$ and $T_1^m \bl v$ are functions {defined} on $\Gamma_h$.
We denote $\wt{\bl v}(\bl x_h):=\bl v\circ M_h(\bl x_h)$, which is also a function on $\Gamma_h$. 
Thus, we have
\begin{equation}
\begin{aligned}\label{T-R}
T^m\bl v-\wt{\bl v}=-R^m \bl v.
\end{aligned}
\end{equation}

For any $e\in\E_h^b$, by choosing local coordinates $(\xi,\eta)$, 
we {define $\Omega_h^e:=\{(\xi,0): 0\leq \xi\leq h_e\}$ with the condition that 
$\eta>0$ in $\Omega_h^e$
(See Fig. \ref{figure1}(a)). 
We present the following three lemmas.}
\begin{lem}\label{Omega_he}
For any $e\in\E_h^b$, one gets
\begin{equation*}
\begin{aligned}
\|\bl v-\wt{\bl v}\|_{0,e}^2\lesssim \delta_h\|\nabla \bl v\|_{0,\Omega_h^e}^2,\qquad \forall \bl v\in \bl H^1(\Omega).
\end{aligned}
\end{equation*}
\end{lem}
\begin{proof}
 {In the local coordinate system $(\xi,\eta)$, 
we have $\dfrac{\pt}{\pt\hspace{-0.12cm}\bl{\n}_h}=\dfrac{\pt}{\pt\eta}$ on $\Omega_h^e$. 
From the definition of $\wt{\bl v}$, we can then conclude that}
\begin{equation*}
\begin{aligned}
|(\bl v-\wt{\bl v})(\xi,0)|^2&=|\bl v(\xi,0)-\bl v(\xi,\delta_h(\xi))|^2\\
&=\left|\int_0^{\delta_h(\xi)}\frac{\pt \bl v}{\pt\eta}(\xi,\eta)\dif\eta\right|^2\\
&\lesssim \delta_h\int_0^{\delta_h(\xi)}|\frac{\pt \bl v}{\pt\eta}(\xi,\eta)|^2\dif\eta,
\end{aligned}
\end{equation*}
and hence, by integrating with respect to $\xi$, we obtain
\begin{equation*}
\begin{aligned}
\int_e|(\bl v-\wt{\bl v})(\xi,0)|^2\dif \xi&\lesssim\delta_h \int_e\int_0^{\delta_h(\xi)}\left|\frac{\pt \bl v}{\pt\eta}(\xi,\eta)\right|^2\dif\eta\dif\xi\\
&\lesssim\delta_h\|\nabla \bl v\|_{0,\Omega_h^e}^2.
\end{aligned}
\end{equation*}
\end{proof}

\begin{lem}\label{lem:B_K1994}
  (cf. \cite{Bramble_King1994}). For each $e\in\E_h^b$ and $v\in H^1(\Omega \cup \Omega_h)$, one has
  \begin{align}
  \|v\|_{0,\Omega_h^e}&\lesssim \delta_h^{1/2}\|v\|_{0,\wt e}+\delta_h\|\nabla v\|_{0,\Omega_h^e}.\label{BK_Me2Me}
  \end{align}
  Moreover, when $v|_{\pt \Omega}=0$, one has
  \begin{equation}
  \begin{aligned}\label{BK_e2Omeg_e}
  \|v\|_{0,\Omega_h^e}\lesssim \delta_h\|v\|_{1,\Omega_h^e}.
  \end{aligned}
  \end{equation}
  \end{lem}

\begin{lem}\label{Omega_he2K}
(cf. \cite{Yiliu_WG}). For $K\in \T_h^b$ and $q\in P_j(K)$, one has
\begin{equation*}
\begin{aligned}
\sum_{e\subset\pt K\cap\Gamma_h}\|q\|_{0,\Omega_h^e}^2\lesssim h_K\|q\|_{0,K}^2.
\end{aligned}
\end{equation*}
\end{lem}

\section{The finite element discretization }\label{sec4:discrete}
{This section introduces the discrete space and a variational formulation. Then, we analyse the well-posdeness of the discrete problem.}
Let $k\geq 1$ be a given integer. 
{We define} the discrete spaces $V_h$ and $Q_h$ {as follows}
 \begin{equation*}
\begin{aligned}
V_h&=\{\bl v_h\in H(\mathrm{div},\Omega_h): \bl v_h|_K\in RT_k(K),K\in\T_h\},\\
Q_h&=\{q_h\in L^2(\Omega_h): q_h|_K\in P_k(K),K\in\T_h\},\\
\end{aligned}
\end{equation*}
where $RT_k(K):=\bl P_k(K)\oplus \bl xP_k(K)$, {with $ P_k(K)$ denoting 
the space of polynomials on $K$
with degree less than or equal to $k$, and $\bl P_k(K)$ representing the corresponding vector space.
We denote} $Q_{0h}:=Q_{h}\cap L_0^2(\Omega_h)$.

For any $K\in\T_h$, we define the local interpolation operator $I_K:\bl H^s(K)\rightarrow RT_k(K), s>1/2,$
{utilizing} the degrees of freedom (dofs.) of Raviart-Thomas finite element \cite{MFEM}.
\begin{equation}
\left\{
\begin{aligned}\label{Pi_k}
\lal I_K \bl v_h\cdot\bl\n_h,\phi_k \ral_e&=\lal  \bl v_h\cdot\bl\n_h,\phi_k \ral_e,\quad &\forall\,&\phi_k\in P_k(e),\,\forall\,e\subset\pt K,\\
( I_K \bl v_h,\bl\psi_{k-1})_K&=(\bl v_h,\bl\psi_{k-1})_K,\quad &\forall\,&\bl\psi_{k-1}\in \bl P_{k-1}(K).
\end{aligned}
\right.
\end{equation}

Define the $L^2$-orthogonal projection $\Pi_k^{0}:L^2(\Omega_h)\rightarrow Q_h$ such that
for any $K\in \T_h$, {the restriction} $\Pi_k^{0,K}:=\Pi_k^0|_K$ satisfies
$$(\Pi_k^{0,K}\phi,q_h)_K=(\phi,q_h)_K,\qquad\forall\,q_h\in Q_h.$$
For every $e\in\E_h$, let $\Pi_k^{0,e}$ {denote the} $L^2$-orthogonal projection onto $P_k(e)$.

{Next, we state the approximation results of the nodal interpolant and $L^2$-orthogonal projections.
The derivation of the following results is
standard \cite{book:RT_app} and \cite{FEM_Brenner}.}

\begin{lem}\label{lem:interpolation_app}
(cf. \cite{book:RT_app}). Let $m$ and $k$ be nonnegative integers such that $0\leq m\leq k+1$. 
Then 
\[
|\bl v-I_K \bl v|_{m,K}\lesssim h^{k-m+1}|\bl v|_{k+1,K}, \qquad \forall\,\bl v\in \bl H^{k+1}(K).
\] 
\end{lem}
{
\begin{lem}\label{lem:orthogonal projection_app}
  (cf. \cite{FEM_Brenner}).
  Let $m$ and $k$ be nonnegative integers such that $0\leq m\leq k+1$. 
  Then 
\begin{align*}
  |v-\Pi_k^{0,K} v|_{m,K}&\lesssim h^{k-m+1}|v|_{k+1,K}, \qquad &\forall\,& v\in H^{k+1}(K),\\
  |v-\Pi_k^{0,e}v|_{m,e}&\lesssim h^{k-m+1}|v|_{k+1,e}, \qquad &\forall\,& v\in H^{k+1}(e), e\subset \pt K.
  \end{align*}
\end{lem}}

\begin{lem}\label{lem:extension} (cf. \cite{book:Extension}).
  {Let $\Omega$ be a Lipschitz domain in $\bbR^2$ and $s\in\bbR$, with $s\geq 0$. 
  Then there exists an extension operator $E: H^s(\Omega)\to H^s(\bbR^2)$ such that}
    \begin{align*}
      E v|_{\Omega}= v,\quad \|E v\|_{s,\bbR^2}\lesssim \|v\|_{s,\Omega},\qquad \forall v\in H^s(\Omega),
    \end{align*}
    where the hidden constant depends on $s$ but not on the diameter of $\Omega$. 
    {Additionally, we have}
    $$\|Ev\|_{s,\Omega_{h}}\leq\|E v\|_{s,\bbR^2}\lesssim \|Ev\|_{s,\Omega}.$$
  
    \end{lem}
    
    For brevity{,} we will also denote {extended functions} by $v^E=E v$ and $f^E=E f$ on $\bbR^2$.
   
\subsection{A variational formulation}
Before deriving a weak formulation, 
it is pointed out that $\bl u^E+\nabla p^E$, $\mathrm{div}\,\bl u^E-f^E$ 
are not {generally} equal to zero on $\Omega_h\backslash \Omega$ when $\Omega_h\nsubseteq \Omega$. 

Next, we multiply the first equation of (\ref{primal_problem}) by an arbitrary $\bl v_h\in V_h$
and {integrate} over $\Omega_h$.
{This leads to the following result}
\begin{equation}
\begin{aligned}\label{eq:deduce_form}
(\bl u^E,\bl v_h)_{\Omega_h}-(p^E,\mathrm{div}\, \bl v_h)_{\Omega_h}+\sum_{e\in\E_h^b}\lal\bl v_h\cdot\bl\n_h,p^E\ral_e=(\bl u^E+\nabla p^E,\bl v_h)_{\Omega_h},
\end{aligned}
\end{equation}
where we have used the fact that $\bl v_h\in V_h \subset H(\mathrm{div})$ implies that $\bl v_h\cdot\bl\n_h$ is continuous across each interior edge. 
Next, by (\ref{T-R}), one has
\begin{equation}
  \begin{aligned}\label{Taylor_ex}
(T^k\bl u^E+R^k \bl u^E)\cdot\wt{\bl\n}=\wt{g}_N, \qquad\text{on}\,\Gamma_h,
\end{aligned}
\end{equation}
where $\wt{g}_N$ is a pull-back of the Neumann boundary data $g_N$ from $\Gamma$ to $\Gamma_h$.
Then, the equality (\ref{eq:deduce_form}) can be rewritten {as}
\begin{equation*}
\begin{aligned}\label{eq:weak_1st}
(\bl u^E,\bl v_h)_{\Omega_h}&+\sum_{e\in\E_h^b}\lal h_K^{-1}T^k\bl u^E\cdot\wt{\bl\n}, T^k\bl v_h\cdot\wt{\bl\n}\ral_e-(p^E,\mathrm{div}\, \bl v_h)_{\Omega_h}\\
&+\sum_{e\in\E_h^b}\lal\bl v_h\cdot\bl\n_h,p^E\ral_e=\sum_{e\in\E_h^b}\lal h_K^{-1}(\wt{g}_N-R^k\bl u^E\cdot\wt{\bl\n}), T^k\bl v_h\cdot\wt{\bl\n}\ral_e+(\bl u^E+\nabla p^E,\bl v_h)_{\Omega_h}.
\end{aligned}
\end{equation*}

The second equation of (\ref{primal_problem}) can be derived by testing any $q_h\in Q_h$ over $\Omega_h$
\begin{equation*}
\begin{aligned}\label{weak_2nd}
(\mathrm{div}\, \bl u^E,q_h)_{\Omega_h}-(f^E,q_h)_{\Omega_h}=(\mathrm{div}\, \bl u^E-f^E,q_h)_{\Omega_h},
\end{aligned}
\end{equation*}
where we also {note that} $f^E=(\mathrm{div}\, \bl u)^E$, and $\mathrm{div}\, \bl u^E\neq (\mathrm{div}\, \bl u)^E$ on $\Omega_h\backslash \Omega$.

{To simplify the notation, for any $u$, $v \in V_h$ and $p\in Q_h$, we define the following bilinear forms}
\begin{equation}
\begin{aligned}\label{bilinear forms}
a_h(\bl u,\bl v):&=(\bl u,\bl v)_{\Omega_h}+\sum_{e\in\E_h^b}\lal h_K^{-1}T^k\bl u\cdot\wt{\bl\n}, T^k\bl v\cdot\wt{\bl\n}\ral_e,\\
b_{h1}(\bl v,p):&=-(\mathrm{div}\, \bl v,p)_{\Omega_h}+\sum_{e\in\E_h^b} \lal\bl v\cdot\bl\n_h,p\ral_e,\\
b_{h0}(\bl v,p):&=-(\mathrm{div}\, \bl v,p)_{\Omega_h}.\\
\end{aligned}
\end{equation}

Then we can present the discrete weak formulation: {\it Find $(\bl u_h,p_h)\in V_h\times {Q_{0h}}$ such that}
\begin{equation}
\left\{
\begin{aligned}\label{Wh1}
a_h(\bl u_h,\bl v_h)+b_{h1}(\bl v_h,p_h)&=\sum_{e\in\E_h^b}\lal h_K^{-1}\wt{g}_N, T^k\bl v_h\cdot\wt{\bl\n}\ral_e, \quad &\forall&\bl v_h\in V_h,\\
b_{h0}(\bl u_h,q_h)&=-(f^E,q_h)_{\Omega_h},  \quad &\forall &q_h\in {Q_{0h}}.
\end{aligned}
\right.
\end{equation}

{Note that a compatibility condition, as stated in (\ref{g_comp_cond}), is necessary for the well-posedness of the continuous problem outlined in (\ref{primal_problem}).
It is important to highlight that the compatibility mechanism for discrete problem works differently, 
with the key being to ensure
$$b_{h1}(\bl v_h,1)=0 \quad \text{and}\quad b_{h0}(\bl u_h,q_h)=-(f^E,q_h),\quad\forall q_h\in Q_{0h}.$$}
{Indeed, it becomes evident through integration by parts that}
$$b_{h1}(\bl v_h, 1)=0, \qquad \forall \bl v_h\in V_h.$$
For any $q_h=\mathrm{const}$, it holds
\begin{equation}
\begin{aligned}\label{compatibility_condition}
(\mathrm{div}\, \bl u_h,q_h)_{\Omega_h}-\sum_{e\in\E_h^b}\lal\bl u_h\cdot\bl\n_h, \bar q_h\ral_e
=(f^E,q_h)_{\Omega_h}-(f^E,\bar q_h)_{\Omega_h},
\end{aligned}
\end{equation}
where $\bar q_h=\frac{1}{|\Omega_h|}\int_{\Omega_h} q_h\dif x$. Moreover, for any $q_h\in Q_{0h}$, it implies $\bar q_h=0$.
Thus, for any $q_h\in Q_{0h}$, the equality (\ref{compatibility_condition}) coincides with the second equality in problem (\ref{Wh1}).
In this way, it demonstrates that (\ref{Wh1}) is now `compatible' in the traditional sense, as it also holds for $q_h\equiv \mathrm{const}.$

For finite-dimensional problems, the effect of the `compatibility' condition may reflect in the implementation procedure.
Actually, it is inconvenient to directly compute a basis function for $Q_{0h}$, which requires the basis functions to be mean-value free on $\Omega_h$.
One usually first computes basis functions for $Q_h$ to form a stiffness matrix $A$ and then modifies 
$A$ by adding a row and a column such that the discrete solution $p_h\in Q_{0h}$.

Define two mesh-dependent norms as follows
\begin{equation*}
\begin{aligned}
\|\bl v_h\|_{0,h}&=\left(\|\bl v_h\|_{0,\Omega_h}^2+\sum_{e\in\E_h^b}\|h_K^{-1/2}T^k\bl v_h\cdot\wt{\bl\n}\|_{0,e}^2\right)^{1/2},\quad &\forall&\bl v_h\in V_h,\\
\|q_h\|_{1,h}&=\left(\sum_{K\in\T_h}\|\nabla q_h\|_{0,K}^2+\sum_{e\in\E_h^o}h^{-1}\|[q_h]\|_{0,e}^2\right)^{1/2},\quad &\forall&q_h\in Q_{0h}.
\end{aligned}
\end{equation*}

\subsection{Well-posedness}
In this subsection, we {discuss} the well-posedness of the discrete problem (\ref{Wh1}). First, we list some relevant lemmas.
\begin{lem}\label{lem:T1_bounded}
For any $\bl v_h\in V_h$, we have
\begin{align}
\sum_{e\in\E_h^b}\|h_K^{-1/2} T_1^m\bl v_h\|_{0,e}^2&\lesssim \|\bl v_h\|_{0,\Omega_h}^2, \label{T1_bound}\\
\sum_{e\in\E_h^b}\|h_K^{-1/2} T^m\bl v_h\|_{0,e}^2&{\lesssim} h^{-2}\|\bl v_h\|_{0,\Omega_h}^2.\label{T0_bounded}
\end{align}
\end{lem}

\begin{proof}
From the definition in (\ref{T1}), 
{along with the trace and inverse inequalities from Lemmas \ref{lem:TRACE} and \ref{lem:inverse}, we obtain}
\begin{equation*}
\begin{aligned}
  \sum_{e\in\E_h^b}\|h_K^{-1/2} T_1^m\bl v_h\|_{0,e}^2&
  =\sum_{e\in\E_h^b}h_K^{-1}\left\|\sum_{j=1}^{m}\frac{\delta_h^j}{j!}\partial_{\bl\nu_h}^j  \bl v_h\right\|_{0,e}^2\\
  &\lesssim \sum_{e\in\E_h^b}\sum_{j=1}^{m} h^{-1} \delta_h^{2j}\|\partial_{\bl\nu_h}^j  \bl v_h\|_{0,e}^2\\
  &\lesssim \sum_{K\in\T_h}\sum_{j=1}^{m} \left(\frac{ \delta_h}{h}\right)^{2j}h^{-2}\|\bl v_h\|_{0,K}^2\\
& \lesssim \sum_{K\in\T_h}\left(\frac{\delta_h}{h}\right)^{2}h^{-2}\|\bl v_h\|_{0,K}^2\\
& \lesssim \|\bl v_h\|_{0,\Omega_h}^2,
\end{aligned}
\end{equation*}
where we have utilized the fact that $\delta_h\lesssim h^2$ in the last inequality. 
Thus, the inequality (\ref{T1_bound}) is proved.
Similarly, we can prove (\ref{T0_bounded}).
\end{proof}

\begin{lem}\label{lem:vn_phi}
  For any $\bl v_h\in V_h$, we have
\begin{equation*}
\begin{aligned}
\sum_{e\in\E_h^b}\lal\bl v_h\cdot\bl\n_h,\phi\ral_e
\lesssim \left(\sum_{e\in\E_h^b} h\|\phi\|_{0,e}^2\right)^{1/2}\|\bl v_h\|_{0,h},
\quad \forall \phi\in{\underset{e\in\E_h^b}{\Pi} L^2(e)}.
\end{aligned}
\end{equation*}
\end{lem}

\begin{proof}
By the definition in (\ref{T1}), notice that $\bl v_h=T^k\bl v_h-T_1^k\bl v_h$, one gets
\begin{equation*}
\begin{aligned}
\sum_{e\in\E_h^b} \lal\bl v_h\cdot\bl\n_h,\phi\ral_e=&\sum_{e\in\E_h^b} (\lal\bl v_h\cdot\wt{\bl\n},\phi\ral_e+\lal\bl v_h\cdot(\bl\n_h-\wt{\bl\n}),\phi\ral_e)\\
=&\sum_{e\in\E_h^b} (\lal T^k\bl v_h\cdot\wt{\bl\n},\phi\ral_e-\lal T_1^k\bl v_h\cdot\wt{\bl\n},\phi\ral_e)+\sum_{e\in\E_h^b}\lal\bl v_h\cdot(\bl\n_h-\wt{\bl\n}),\phi\ral_e.
\end{aligned}
\end{equation*}
By the Schwarz inequality and (\ref{deta}), one has
\begin{equation*}
\begin{aligned}
\sum_{e\in\E_h^b} \lal\bl v_h\cdot\bl\n_h,\phi\ral_e \leq &\sum_{e\in\E_h^b} (h_K^{-1/2}\| T^k\bl v_h\cdot\wt{\bl\n}\|_{0,e}
+h_K^{-1/2}\|T_1^k\bl v_h\cdot\wt{\bl\n}\|_{0,e}+h_K^{1/2}\|\bl v_h\|_{0,e})h_K^{1/2}\|\phi\|_{0,e}.\\
\end{aligned}
\end{equation*}
Next, according to Lemma \ref{lem:T1_bounded} and the inverse inequality, it is no hard to obtain
\begin{equation*}
\begin{aligned}
\sum_{e\in\E_h^b} \lal\bl v_h\cdot\bl\n_h,\phi\ral_e &\lesssim \left(\sum_{e\in\E_h^b} \|h_K^{-1/2} T^k\bl v_h\cdot\wt{\bl\n}\|_{0,e}^2+\|\bl v_h\|_{0,\Omega_h}^2\right)^{1/2}
\left(\sum_{e\in\E_h^b} h\|\phi\|_{0,e}^2\right)^{1/2}\\
&\lesssim \left(\sum_{e\in\E_h^b} h\|\phi\|_{0,e}^2\right)^{1/2} \|\bl v_h\|_{0,h}.
\end{aligned}
\end{equation*}
\end{proof}

\begin{cro}\label{cro:vn_ph}
  For any $\bl v_h\in V_h$, we have
\begin{equation*}
\begin{aligned}
\sum_{e\in\E_h^b}\lal\bl v_h\cdot\bl\n_h,q_h\ral_e\lesssim \|\bl v_h\|_{0,h}\|q_h\|_{1,h},\qquad \forall q_h\in Q_{0h}.
\end{aligned}
\end{equation*}
\end{cro}

\begin{proof}
Replacing $\phi$ of Lemma \ref{lem:vn_phi} by $q_h$, we have
\begin{equation*}
\begin{aligned}
\sum_{e\in\E_h^b}\lal\bl v_h\cdot\bl\n_h,q_h\ral_e\lesssim \left(\sum_{e\in\E_h^b} h\|q_h\|_{0,e}^2\right)^{1/2}\|\bl v_h\|_{0,h}.
\end{aligned}
\end{equation*}
By the trace and inverse inequalities in Lemmas \ref{lem:TRACE} and \ref{lem:inverse}, we obtain
\begin{equation}
\begin{aligned}\label{J2}
\sum_{e\in\E_h^b}\lal\bl v_h\cdot\bl\n_h,q_h\ral_e \lesssim \|\bl v_h\|_{0,h}\|q_h\|_{0,\Omega_h}.
\end{aligned}
\end{equation}
Notice that $q_h\in Q_{0h}$, which implies $\int_{\Omega_h} q_h\dif x=0$.
Thus, by the Poncar\'{e}-Friedrichs inequality in Lemma \ref{lem:Poincare_inequality}, yielding
\begin{equation}
\begin{aligned}\label{poincare_Fridrichs}
\|q_h\|_{0,\Omega_h}^2\lesssim \|q_h\|_{1,h}^2.
\end{aligned}
\end{equation}
Substituting (\ref{poincare_Fridrichs}) into (\ref{J2}), we obtain the desired result.
\end{proof}

We are now in the position to state the main results of this subsection:
\begin{lem}\label{lem:bounded}
For any $\bl u_h,\bl v_h\in V_h,\, q_h\in Q_{0h}$, we have
\begin{align}
|a_h(\bl u_h,\bl v_h)|&\leq \|\bl u_h\|_{0,h} \|\bl v_h\|_{0,h},&\qquad |a_h(\bl v_h,\bl v_h)|&{=} \|\bl v_h\|_{0,h}^2,\label{ah_bound} \\ 
|b_{h1}(\bl v_h,q_h)|& \lesssim \|\bl v_h\|_{0,h} \|q_h\|_{1,h},&\qquad |b_{h0}(\bl v_h,q_h)|&\lesssim\|\bl v_h\|_{0,h} \|q_h\|_{1,h}.\label{bh_bound}
\end{align}
\end{lem}

\begin{proof}
From the definition of $\|\cdot\|_{0,h}$, it is evident to complete the proof of (\ref{ah_bound}).
Then, we only need to prove (\ref{bh_bound}).
By integration by parts and the Schwarz inequality, we obtain
\begin{equation*}
\begin{aligned}
b_{h1}(\bl v_h,q_h)&=-\sum_{K\in\T_h}(\mathrm{div}\, \bl v_h,q_h)_K+\sum_{e\in\E_h^b} \lal\bl v_h\cdot\bl\n_h,q_h\ral_e\\
&=\sum_{K\in\T_h}(\bl v_h,\nabla q_h)_K-\sum_{e\in\E_h^o} \lal\bl v_h\cdot\bl\n_h,[q_h]\ral_e\\
&\leq \sum_{K\in\T_h}\|\bl v_h\|_{0,K}\|\nabla q_h\|_{0,K}+\sum_{e\in\E_h^o} h^{1/2}\|\bl v_h\cdot\bl\n_h\|_{0,e}h^{-1/2}\|[q_h]\|_{0,e},
\end{aligned}
\end{equation*}
using the trace and inverse inequalities in Lemmas \ref{lem:TRACE}, \ref{lem:inverse}, we get
\begin{equation}
\begin{aligned}\label{inverse_inequality}
h^{1/2}\|\bl v_h\cdot\bl\n_h\|_{0,e}\leq h^{1/2}\|\bl v_h\|_{0,e}\lesssim \|\bl v_h\|_{0,K_1\cup K_2}, \quad e\in\E_h^o,\, e=K_1\cap K_2.
\end{aligned}
\end{equation}
Then
\begin{equation*}
\begin{aligned}
|b_{h1}(\bl v_h,q_h)|& \lesssim \|\bl v_h\|_{0,h} \|q_h\|_{1,h}.
\end{aligned}
\end{equation*}
Similarly, by Corollary \ref{cro:vn_ph}, the bound of $b_{h0}(\bl v_h,q_h)$ can be estimated as 
\begin{equation}
\begin{aligned}\label{b0}
|b_{h0}(\bl v_h,q_h)|&=|\sum_{K\in\T_h}(\bl v_h,\nabla q_h)_K-\sum_{e\in\E_h^o} \lal\bl v_h\cdot\bl\n_h,[q_h]\ral_e
-\sum_{e\in\E_h^b} \lal\bl v_h\cdot\bl\n_h,q_h\ral_e|\\
&\lesssim \|\bl v_h\|_{0,h}\|q_h\|_{1,h}+
|\sum_{e\in\E_h^b} \lal\bl v_h\cdot\bl\n_h,q_h\ral_e|\\
&\lesssim \|\bl v_h\|_{0,h}\|q_h\|_{1,h}.
\end{aligned}
\end{equation}

Thus, we complete the proof of this lemma.
\end{proof}

\begin{lem}\label{lem:LBB}
  (Inf-Sup). For all $q_h\in Q_{0h}$, it holds
\begin{equation*}
\begin{aligned}\label{LBB12}
\sup_{\bl v_h\in V_h}\frac{b_{h1}(\bl v_h,q_h)}{\|\bl v_h\|_{0,h}}\gtrsim  \|q_h\|_{1,h},\qquad
\sup_{\bl v_h\in V_h}\frac{b_{h0}(\bl v_h,q_h)}{\|\bl v_h\|_{0,h}}\gtrsim  \|q_h\|_{1,h}.
\end{aligned}
\end{equation*}
\end{lem}

\begin{proof}
For an arbitrary $q_h\in Q_{0h}$, we construct $\bl v_h$
  using the degrees of freedom of the Raviart-Thomas space as follows:
\begin{align}
\lal\bl v_h\cdot\bl\n_h,\phi_k\ral_e&=-h^{-1}\lal[q_h],\phi_k\ral_e, \quad &\forall &\phi_k\in P_k(e),\forall \,e\in\E_h^o,\\
\lal\bl v_h\cdot\bl\n_h,\phi_k\ral_e&=0, \quad &\forall &\phi_k\in P_k(e),\forall e\in\E_h^b,\label{construct_Eb}\\
(\bl v_h,\bl\psi_{k-1})_K&=(\nabla q_h,\bl\psi_{k-1})_K, \quad &\forall &\bl\psi_{k-1}\in \bl P_{k-1}(K),
\forall \,K\in\T_h.
\end{align}
From reference \cite{Puppi_mixed}, it is straightforward to obtain
\begin{equation}
\begin{aligned}\label{b1_bound}
b_{h1}(\bl v_h,q_h)={\|q_h\|_{1,h}^2}, \qquad \|\bl v_h\|_{0,\Omega_h}\lesssim \|q_h\|_{1,h}.
\end{aligned}
\end{equation}
Recalling the definition of $\|\cdot\|_{0,h}$, it suffices to prove the inequality $$\sum_{e\in\E_h^b} \|h_K^{-1/2} T^k\bl v_h\cdot\wt{\bl\n}\|_{0,e}\lesssim \|q_h\|_{1,h}.$$
Indeed, for each $e\in\E_h^b$ and $\bl v_h\in V_h$ implies $\bl v_h\cdot\bl\n_h|_e\in P_k(e)$.
Then replacing $\phi_k$ of (\ref{construct_Eb}) by $\bl v_h\cdot\bl\n_h$, we have
\begin{equation}
\begin{aligned}\label{v*n=0}
\bl v_h\cdot\bl\n_h|_e=0,\qquad \forall e\in \E_h^b.
\end{aligned}
\end{equation}
{Furthermore, by {the triangle inequality}, it has}
\begin{equation}
\begin{aligned}\label{E}
\sum_{e\in\E_h^b} \|h_K^{-1/2} &T^k\bl v_h\cdot\wt{\bl\n}\|_{0,e}^2\lesssim\sum_{e\in\E_h^b} h_K^{-1}(\| T^k\bl v_h\cdot\bl\n_h\|_{0,e}^2+\sum_{e\in\E_h^b} h_K^{-1}\| T^k\bl v_h\cdot\wt{\bl\n}-\bl\n_h\|_{0,e}^2)\\
&\lesssim \sum_{e\in\E_h^b} h_K^{-1}(\| T_1^k\bl v_h\cdot\bl\n_h\|_{0,e}^2+\| T^k\bl v_h\cdot(\wt{\bl\n}-\bl\n_h)\|_{0,e}^2).
\end{aligned}
\end{equation}
Applying Lemma \ref{lem:T1_bounded} and  (\ref{deta}), we obtain
\begin{equation}
\begin{aligned}\label{E_bound0}
\sum_{e\in\E_h^b} h_K^{-1}\| T_1^k\bl v_h\cdot\bl\n_h\|_{0,e}^2&\lesssim \|\bl v_h\|_{0,\Omega_h}^2,
\end{aligned}
\end{equation}
and
\begin{equation}
\begin{aligned}\label{E_bound}
\sum_{e\in\E_h^b} h_K^{-1}\| T^k\bl v_h\cdot(\wt{\bl\n}-\bl\n_h)\|_{0,e}^2&\lesssim h^{-2}\|\bl v_h\|_{0,\Omega_h}^2\|\wt{\bl\n}-\bl\n_h\|_{L^\infty(\Gamma_h)}^2\\
&\lesssim \|\bl v_h\|_{0,\Omega_h}^2.
\end{aligned}
\end{equation}
Finally, combining (\ref{b1_bound})-(\ref{E_bound}), we establish the first inequality of this lemma.

{Similarly} to the estimate of $b_{h1}(\bl v_h,q_h)$, by integration by parts and (\ref{v*n=0}), one gets
\begin{equation*}
\begin{aligned}
b_{h0}(\bl v_h,q_h)&=\sum_{K\in\T_h}(\bl v_h,\nabla q_h)_K-\sum_{e\in\E_h^o}\lal\bl v_h\cdot\bl\n_h,[q_h]\ral_e\\
&=\sum_{K\in\T_h}\|\nabla q_h\|_{0,K}^2+\sum_{e\in\E_h^o}h^{-1}\|[q_h]\|_{0,e}^2=\|q_h\|_{1,h}^2.
\end{aligned}
\end{equation*}
Thus, we complete the proof of this lemma.
\end{proof}

According to  the Brezzi theorem, the discrete problem (\ref{Wh1}) admits a unique solution.

\section{Error analysis}\label{sec5:error}
{In this section, we} will estimate the errors in mesh-dependent norm and $L^2$-norm.
The stability estimates of $a_h(\cdot,\cdot), b_{h0}(\cdot,\cdot)$ and $b_{h1}(\cdot,\cdot)$ imply
\begin{equation}
\begin{aligned}\label{B_h_bound}
\|\bl\sigma_h,\zeta_h\|_H\lesssim \sup_{(\bl v_h, q_h)\in V_h\times Q_{0h}}\frac{B_h((\bl\sigma_h,\zeta_h),(\bl v_h, q_h))}{\|\bl v_h, q_h\|_H},
\qquad \forall\,(\bl\sigma_h,\zeta_h)\in V_h\times Q_{0h},
\end{aligned}
\end{equation}
where
\begin{equation*}
\begin{aligned}
B_h((\bl\sigma_h,\zeta_h),(\bl v_h, q_h))&:=a_h(\bl\sigma_h,\bl v_h)+b_{h1}(\bl v_h,\zeta_h)+b_{h0}(\bl\sigma_h,q_h),\\
\|\bl\sigma_h,\zeta_h\|_H&:=(\|\bl\sigma_h\|_{0,h}^2+\|\zeta_h\|_{1,h}^2)^{1/2}.
\end{aligned}
\end{equation*}

To obtain the error estimates, we need the following lemma.
\begin{lem}\label{lem:app_Th}
Assume that $\bl w\in W^{k+1,\infty}(\Omega)\cap\bl H^{r+1}(\Omega)$, $\bl w^E$ is the extended function of $\bl w$, for any $\bl v_h\in V_h$, the interpolation operator $\bl w_I$ of $\bl w^E$ satisfies
\begin{equation*}
\begin{aligned}
\sum_{e\in\E_h^b}\lal h_K^{-1} T^k (\bl w^E-\bl w_I)\cdot\wt{\bl\n},T^k\bl v_h\cdot\wt{\bl\n}\ral_e
\lesssim \left(h^{k+1/2}\|\bl w^E\|_{k+1,\infty,\Omega_h}+h^{r+1}|\bl w^E|_{r+1,\Omega_h}\right)\|\bl v_h\|_{0,h}.
\end{aligned}
\end{equation*}
\end{lem}

\begin{proof}
By applying the Schwarz inequality and (\ref{deta}), we infer that
\begin{equation*}
\begin{aligned}
   \sum_{e\in\E_h^b}&\lal h_K^{-1}T^k (\bl w^E-\bl w_I)\cdot\wt{\bl\n},T^k\bl v_h\cdot\wt{\bl\n}\ral_e\\
   =&\sum_{e\in\E_h^b}h_K^{-1}\lal T^k (\bl w^E-\bl w_I)\cdot\bl\n_h+T^k (\bl w^E-\bl w_I)\cdot(\wt{\bl\n}-\bl\n_h),T^k\bl v_h\cdot\wt{\bl\n}\ral_e\\
   \lesssim& \sum_{e\in\E_h^b}h_K^{-1}\bigg(\lal (\bl w^E-\bl w_I)\cdot\bl\n_h+T_1 (\bl w^E-\bl w_I)\cdot\bl\n_h,T^k\bl v_h\cdot\wt{\bl\n}\ral_e\\
   &\qquad+\|\wt{\bl\n}-\bl\n_h\|_{\infty,\Gamma_h}\|T^k (\bl w^E-\bl w_I)\|_{0,e}\|T^k \bl v_h\cdot\wt{\bl\n}\|_{0,e}\bigg)\\
   \lesssim &\left(\sum_{e\in\E_h^b}h^{-1}(\| (\bl w^E-\bl w_I)\cdot\bl\n_h\|_{0,e}^2+\|T_1^k (\bl w^E-\bl w_I)\|_{0,e}^2+h^2\|T^k (\bl w^E-\bl w_I)\|_{0,e}^2)\right)^{1/2}\|\bl v_h\|_{0,h}.
\end{aligned}
\end{equation*}
From the definition of interpolation, we observe that $\bl w_I\cdot\bl\n_h|_e=\Pi_k^{0,e}(\bl w\cdot\bl\n_h)$.
Then, by Lemma \ref{lem:interpolation_app}, we deduce that 
\begin{equation*}
\begin{aligned}\label{infty}
\sum_{e\in\E_h^b}h^{-1}\|(\bl w^E-\bl w_I)\cdot\bl\n_h\|_{0,e}^2&\lesssim h^{2k+1}\sum_{e\in\E_h^b}\|\bl w^E\cdot\bl\n_h\|_{k+1,e}^2\\
&\lesssim h^{2k+1}\sum_{e\in\E_h^b}|e|\|\bl w^E\|_{k+1,\infty,e}^2\\
&\lesssim h^{2k+1}\|\bl w^E\|_{k+1,\infty,\Omega_h}^2.
\end{aligned}
\end{equation*}
By the trace inequality in Lemmas \ref{lem:TRACE} and \ref{lem:interpolation_app}, we have
\begin{equation*}
\begin{aligned}\label{I_1_bounded}
\sum_{e\in\E_h^b}h^{-1}(\|T_1^k (\bl w^E-\bl w_I)\|_{0,e}^2+h^2\|T^k (\bl w^E-\bl w_I)\|_{0,e}^2)\lesssim h^{2(r+1)}|\bl w^E|_{r+1,\Omega_h}^2.
\end{aligned}
\end{equation*}
Combining the above, we complete the proof of this lemma.
\end{proof}

\begin{thm}\label{thm:energy_err}
  Let $(\bl u,p)\in \bl W^{k+1,\infty}(\Omega)\cap\bl H^{\max\{r+1,\lceil l/2\rceil \}}(\Omega)\times H^{\max\{t+1,\lceil l/2\rceil \}}(\Omega)$ be the solution of problem (\ref{primal_problem})
  and $\bl u^E\in \bl W^{k+1,\infty}(\Omega\cup\Omega_h)\cap\bl H^{\max\{r+1,\lceil l/2\rceil \}}(\Omega\cup\Omega_h),p^E\in H^{\max\{t+1,\lceil l/2\rceil \}}(\Omega\cup\Omega_h)$ are extended functions of $\bl u$ and $p$.
  Let $(\bl u_h,p_h)\in V_h\times Q_{0h}$ be the discrete solution of (\ref{Wh1}), $s:=\min\{r,t,k\}$, it holds
\begin{equation*}
\begin{aligned}
\|\bl u_h-&\bl u_I\|_{0,h}+\|p_h-p_I\|_{1,h}\lesssim~ h^{s+1}(|\bl u|_{r+1,\Omega}+|p|_{t+1,\Omega})+h^{k+1/2}\|\bl u^E\|_{k+1,\infty,\Omega_h}\\
&+\delta^{k+1}h^{-\frac{1}{2}}M_{k+1}(\bl u^E)+\delta^{l}\bigg(\|D^l(\bl u^E+\nabla p^E)\|_{0,\Omega_h\backslash\Omega}+\|D^l({\mathrm{div}}\,\bl u^E-f^E)\|_{0,\Omega_h\backslash \Omega}\bigg),
\end{aligned}
\end{equation*}
where $\bl u_I$ and $p_I$ denote the interpolation and $L^2$ projection of $\bl u^E$ and $p^E$, respectively.
\end{thm}

\begin{proof}
From (\ref{B_h_bound}), one has
\begin{equation}
\begin{aligned}\label{err_1}
\|\bl u_h-\bl u_I,p_h-p_I\|_H\lesssim \sup_{(\bl v_h, q_h)\in V_h\times Q_{0h}}\frac{B_h((\bl u_h-\bl u_I,p_h-p_I),(\bl v_h, q_h))}{\|\bl v_h, q_h\|_H}.
\end{aligned}
\end{equation}
By adding and subtracting $\bl u^E$ and $p^E$, one obtains
\begin{equation}
\begin{aligned}\label{err_2}
B_h((\bl u_h-\bl u_I,p_h-p_I),(\bl v_h, q_h))=&~a_h(\bl u_h-\bl u^E,\bl v_h)+b_{h1}(\bl v_h,p_h-p^E)+b_{h0}(\bl u_h-\bl u^E,q_h)\\
&+a_h(\bl u^E-\bl u_I,\bl v_h)+b_{h1}(\bl v_h,p^E-p_I)+b_{h0}(\bl u^E-\bl u_I,q_h)\\
:= &~ E_R+E_h.
\end{aligned}
\end{equation}
By the primal problem (\ref{primal_problem}), the discrete problem (\ref{Wh1}), along with equations (\ref{Taylor_ex}) and (\ref{BK_e2Omeg_e}), the consistency error is expressed as follows
\begin{equation}
\begin{aligned}\label{err_3}
E_R&=a_h(\bl u_h-\bl u^E,\bl v_h)+b_{h1}(\bl v_h,p_h-p^E)+b_{h0}(\bl u_h-\bl u^E,q_h)\\
&=\sum_{e\in\E_h^b}h^{-1}\lal\wt{g}_N-T^k\bl u^E\cdot\wt{\bl\n}, T^k\bl v_h\cdot\wt{\bl\n}\ral_e+(\bl u^E-\nabla p^E,\bl v_h)_{\Omega_h}+(\mathrm{div}\,u^E-f^E,q_h)_{\Omega_h}\\
&\lesssim \delta^{k+1}h^{-1/2}M_{k+1}(\bl u^E)\|\bl v_h\|_{0,h}+\delta^{l}\bigg(\|D^l(\bl u^E+\nabla p^E)\|_{0,\Omega_h\backslash\Omega}+\|D^l(\mathrm{div}\,\bl u^E-f^E)\|_{0,\Omega_h\backslash \Omega}\bigg),
\end{aligned}
\end{equation}
where we have utilized $(\bl u^E-\nabla p^E)|_{\Omega}=0$ and $(\mathrm{div}\,\bl u^E-f^E)|_{\Omega}=0$.
Next, we estimate the remaining approximation terms. 
From the definition of interpolation given in (\ref{Pi_k}), it follows that
\begin{equation*}
\begin{aligned}
b_{h0}(\bl u^E-\bl u_I,q_h)&=-(\mathrm{div}\,(\bl u^E-\bl u_I),q_h)\\
&=\sum_{K\in\T_h}(\bl u^E-\bl u_I,\nabla q_h)_K-\sum_{K\in\T_h}\lal(\bl u^E-\bl u_I)\cdot\bl\n_h, q_h\ral_{\partial K}\\
&=0,
\end{aligned}
\end{equation*}
and
\begin{equation*}
\begin{aligned}
b_{h1}(\bl v_h,p^E-p_I)&=-\sum_{K\in\T_h}(\mathrm{div}\,\bl v_h,p^E-p_I)_K+\sum_{e\in\E_h^b}\lal\bl v_h\cdot\bl\n_h, p^E-p_I\ral_e\\
&=\sum_{e\in\E_h^b}\lal\bl v_h\cdot\bl\n_h, p^E-p_I\ral_e.
\end{aligned}
\end{equation*}
Thus, by rearranging the term $E_h$, we deduce
\begin{equation*}
\begin{aligned}\label{err_5}
E_h&= a_h(\bl u^E-\bl u_I,\bl v_h)+b_{h1}(\bl v_h,p^E-p_I)+b_{h0}(\bl u^E-\bl u_I,q_h)\\
&=a_h(\bl u^E-\bl u_I,\bl v_h)+\sum_{e\in\E_h^b}\lal\bl v_h\cdot\bl\n_h, p^E-p_I\ral_e.
\end{aligned}
\end{equation*}
By the Schwarz inequality and Lemmas \ref{lem:app_Th}, \ref{lem:interpolation_app}, we derive
\begin{equation*}
\begin{aligned}
E_h\lesssim \left(h^{r+1}|\bl u^E|_{r+1,\Omega_h}+h^{k+1/2}\|\bl u^E\|_{k+1,\infty,\Omega_h}\right)
\|\bl v_h\|_{0,h}+\sum_{e\in\E_h^b}\lal\bl v_h\cdot\bl\n_h, p^E-p_I\ral_e.
\end{aligned}
\end{equation*}
From Lemma \ref{lem:vn_phi}, it holds
\begin{equation*}
\begin{aligned}
\sum_{e\in\E_h^b}\lal\bl v_h\cdot\bl\n_h, p^E-p_I\ral_e\lesssim \left(\sum_{e\in\E_h^b} h\|p^E-p_I\|_{0,e}^2\right)^{1/2}\|\bl v_h\|_{0,h}.
\end{aligned}
\end{equation*}
Finally, by the trace inequality in Lemma \ref{lem:TRACE} and the approximation property of $L^2$-projection in Lemma \ref{lem:orthogonal projection_app}, one gets
\begin{equation*}
\begin{aligned}\label{I_p_bounded}
\sum_{e\in\E_h^b}\lal\bl v_h\cdot\bl\n_h,p^E-p_I\ral_e\lesssim h^{t+1}|p^E|_{t+1,\Omega_h}\|\bl v_h\|_{0,h}.
\end{aligned}
\end{equation*}
Thus
\begin{equation}
\begin{aligned}\label{err_4}
E_h\lesssim \left(h^{k+1/2}\|\bl u^E\|_{k+1,\infty,\Omega_h}+h^{r+1}|\bl u^E|_{r+1,\Omega_h}+h^{t+1}|p^E|_{t+1,\Omega_h}\right)\|\bl v_h\|_{0,h}.
\end{aligned}
\end{equation}

Combining (\ref{err_1})-(\ref{err_4}) and Lemma \ref{lem:extension}, the theorem is proved.
\end{proof}

\begin{thm}\label{thm:L2err}
Under the same assumption of Theorem \ref{thm:energy_err}, for $s:=\min\{r,t,k\}$, one gets
\begin{equation*}
\begin{aligned}
\|\bl u^E-\bl u_h\|_{0,\Omega_h}&\lesssim h^{s+1}(|\bl u|_{r+1,\Omega}+|p|_{t+1,\Omega})+h^{k+1/2}\|\bl u^E\|_{k+1,\infty,\Omega_h}
+\delta^{k+1}h^{-\frac{1}{2}}M_{k+1}(\bl u^E),\\
\|\nabla (p-p_h)\|_{0,\T_h}&\lesssim h^{s}(|\bl u|_{r+1,\Omega}+|p|_{t+1,\Omega})+h^{k+1/2}\|\bl u^E\|_{k+1,\infty,\Omega_h}
+\delta^{k+1}h^{-\frac{1}{2}}M_{k+1}(\bl u^E).
\end{aligned}
\end{equation*}
\end{thm}

\begin{proof}
Using the triangle inequality and the approximation error of interpolation in Lemma \ref{lem:interpolation_app} yields
\begin{equation*}
\begin{aligned}
\|\bl u^E-\bl u_h\|_{0,\Omega_h}&\leq \|\bl u^E-\bl u_I\|_{0,\Omega_h}+\|\bl u_I-\bl u_h\|_{0,\Omega_h}\\
&\lesssim h^{r+1}|\bl u^E|_{r+1,\Omega_h}+\|\bl u_I-\bl u_h\|_{0,h},\\
\end{aligned}
\end{equation*}
and the approximation error of $L^2$ projection in Lemma \ref{lem:orthogonal projection_app} implies
\begin{equation*}
\begin{aligned}
\|\nabla(p^E-p_h)\|_{0,\T_h}&\leq \|\nabla(p^E-p_I)\|_{0,\T_h}+\|\nabla(p_I-p_h)\|_{0,\T_h}\\
&\lesssim h^{t}|p^E|_{t+1,\Omega_h}+\|p_I-p_h\|_{1,h},
\end{aligned}
\end{equation*}
then, combining the Theorem \ref{thm:energy_err} and Lemma \ref{lem:extension}, we obtain this proof.
\end{proof}

\begin{re}
Under the assumption $\delta\lesssim h^2$, Theorem \ref{thm:L2err} demonstrates a suboptimal convergence rate for the velocity field in the $L^2$-norm
and an optimal order of convergence for pressure field in the $H^1$-norm.
From Tab. \ref{tab_err}, it is observed that the velocity error estimate is $O(h^{k+1/2})$ for $s=k$.
\end{re}

\begin{table}[!htbp]
\caption{\small\textit{The error order of terms in $L^2$-norm under the assumption $\delta\lesssim h^2$.}}
\label{tab_err}
\begin{center}
{\small
\begin{tabular}{c|c|c|c|c|c|c}
\cline{1-7}
$s$  &$h^{s+1}$   &$k$    &$h^{k+1/2}$              &$\delta^{k+1}h^{-\frac{1}{2}}$   &$l$  &$\delta^{l}$\\
1    &$h^2$       &1      &$h^{1.5}$                   &$h^{3.5}$                      &$1$     &$h^2$\\
2    &$h^3$       &2      &$h^{2.5}$                    &$h^{5.5}$                      &$2$     &$h^4$\\
3    &$h^4$       &3      &$h^{3.5}$                   &$h^{7.5}$                    &$3$     &$h^6$\\
\cline{1-7}
\end{tabular}}
\end{center}
\end{table}

\section{Discrete problem without correction}\label{sec6:without_correction}

{The goal of this section is} to verify the necessity of boundary-value correction on curved domains. We will briefly explain the main idea without applying boundary value correction on $\Gamma_h$.
For simplicity, we assume that $g_N=0$ of problem (\ref{primal_problem}). 
Without boundary value correction means that $\bl u^E\cdot\bl \n_h=0$ is coercively imposed on $\Gamma_h$.
Define
$$H^{r}(\mathrm{div},S):=\{\bl v\in \bl H^r(S):\mathrm{div}\,\bl v\in H^r(S)\},$$

The weak formulation can be written as: {\it Find $\bl u_h\in V_{0h}$ and $p_h\in Q_{0h}$ such that}
\begin{equation}
\left\{
\begin{aligned}\label{Without_Ph}
(\bl u_h,\bl v_h)_{\Omega_h}-(\mathrm{div}\,\bl v_h,p_h)_{\Omega_h}&=0, \quad &\forall&\bl v_h\in V_{0h},\\
-(\mathrm{div}\,\bl u_h,q_h)_{\Omega_h}&=-(f^E,q_h)_{\Omega_h},  \quad &\forall &q_h\in Q_{0h},
\end{aligned}
\right.
\end{equation}
where $V_{0h}:=V_h\cap H_0(\mathrm{div},\Omega_h)$.

Due to $f^E=(\mathrm{div}\, \bl u)^E$, and $\mathrm{div}\, \bl u^E\neq (\mathrm{div}\, \bl u)^E$ on $\Omega_h\backslash \Omega$, the error equation can be easily deduce:
\begin{align}\label{err_equation_N}
B_h^*((\bl u^E-\bl u_h,p^E-p_h),(\bl v_h, q_h))&=(f^E-\mathrm{div}\,\bl u^E,q_h)_{\Omega_h},
\end{align}
where
$$B_h^*((\bl w,\phi),(\bl v_h, q_h)):=(\bl w,\bl v_h)_{\Omega_h}-(\mathrm{div}\,\bl v_h,\phi)_{\Omega_h}-(\mathrm{div}\,\bl w,q_h)_{\Omega_h}.$$

Define $\widetilde I_h:H_0(\mathrm{div},\Omega_h)\cap \prod_{K\in\T_h}H^s(K)\rightarrow V_{0h}, s>1/2$ satisfies
\begin{equation}
\left\{
\begin{aligned}\label{pi_without}
\widetilde I_K \bl v\cdot\bl\n_h|_e&=0,\quad &\forall\,&e\in\E_h^b,\\
\lal \widetilde I_K \bl v\cdot\bl\n_h, q_k \ral_e& =\lal \bl v\cdot\bl\n_h, q_k \ral_e,\quad &\forall\, &q_k(e)\in P_k(e),\forall\,e\in\E_h^o,\\
( \widetilde I_K \bl v,\bl q_{k-1})_K&=(\bl v,\bl q_{k-1})_K,\quad &\forall\,&\bl q_{k-1}\in \bl P_{k-1}(K),\,\forall\,K\in\T_h,
\end{aligned}
\right.
\end{equation}
where $\widetilde I_K=\widetilde I_h|_K$. 
Furthermore, we derive
\begin{equation}
\begin{aligned}\label{commutative_diagram}
\mathrm{div}\,\widetilde I_h \bl v=\Pi_k^{0}\mathrm{div}\,\bl v .
\end{aligned}
\end{equation}

Through the definitions $\widetilde I_K$ and $ I_K$ of (\ref{Pi_k}), we obtain the following
\begin{equation}
\begin{aligned}\label{Pi_equal}
I_K\bl v-\widetilde I_K \bl v=&~0,\quad &\forall &K\in\T_h^o,\\
I_K\bl v-\widetilde I_K \bl v=&\sum_{i=0}^{k}\int_e \bl v\cdot \bl \n_h\,q_i\dif s\,\psi_i^e,\quad &\forall& q_i\in P_i(e),e\in\E_h^b,K\in\T_h^b,
\end{aligned}
\end{equation}
where $\{\psi_i^e\}_{i=0}^k$ denotes the basis functions of $RT_k$ on $e\in\E_h^b$.

By the error equation (\ref{err_equation_N}), together with the second Strang's lemma in \cite{FEM_Brenner}, one obtains the following estimate:
\begin{lem}\label{lem:without_errapp}
Let $(\bl u,p)$ be the solution of (\ref{primal_problem}) with $g_N=0$ and $(\bl u_h,p_h)\in V_{0h}\times Q_{0h}$
be the solution of (\ref{Without_Ph}). Then
\begin{equation*}
\begin{aligned}
\|\bl u^E-\bl u_h\|_{H(\mathrm{div},\Omega_h)}+\|p^E-p_h\|_{0,\Omega_h}&\lesssim
\inf_{\bl v_h\in V_{0h}}\|\bl u^E-\bl v_h\|_{H(\mathrm{div},\Omega_h)}+\inf_{q_h\in Q_{0h}}\|p^E-q_h\|_{0,\Omega_h}\\
&\qquad + \|f^E-\mathrm{div}\,\bl u^E\|_{0,\Omega_h}.
\end{aligned}
\end{equation*}
\end{lem}

\begin{thm}\label{thm:without_err}
Let $(\bl u,p)\in H^{\max\{2,r,\lceil l/2\rceil\}}(\mathrm{div},\Omega)\times H^{\max\{t+1,\lceil l/2\rceil\}}(\Omega)$ be the solution of problem (\ref{primal_problem})
and $\bl u^E$, $p^E$ and $f^E$ are the extension functions of $\bl u$, $p$ and $f$. 
Let $(\bl u_h,p_h)\in V_{0h}\times Q_{0h}$ be the discrete solution of (\ref{Without_Ph}), it holds
\begin{equation*}
\begin{aligned}
\|\bl u^E-\bl u_h\|_{H(\mathrm{div},\Omega_h)}+\|p^E-p_h\|_{0,\Omega_h}&\lesssim h^{3/2}|\bl u|_{2,\Omega}+h^{r}\|\mathrm{div}\,\bl u^E\|_{r,\Omega_h}+h^{t+1}|p|_{t+1,\Omega}\\
&\qquad + \delta^l\|D^l(f^E-\mathrm{div}\,\bl u^E)\|_{0,\Omega_h\backslash \Omega}.
\end{aligned}
\end{equation*}
\end{thm}

\begin{proof}
Using Inequality (\ref{BK_e2Omeg_e}) to the error equation~\eqref{err_equation_N}, it holds
\begin{equation*}
\begin{aligned}
\|f^E-\mathrm{div}\,\bl u^E\|_{0,\Omega_h}\lesssim \delta^l\|D^l(f^E-\mathrm{div}\,\bl u^E)\|_{0,\Omega_h\backslash \Omega}.
\end{aligned}
\end{equation*}
From (\ref{commutative_diagram}) and the property of $L^2$-projection, we obtain
\begin{equation*}
\begin{aligned}
\|\mathrm{div}(\bl u^E-\widetilde I_h \bl u^E)\|_{0,\Omega_h}&=\|\mathrm{div}\,\bl u^E-\Pi_k^{0}\mathrm{div}\,\bl u^E\|_{0,\Omega_h} \\
&\lesssim h^r\|\mathrm{div}\,\bl u^E\|_{r,\Omega_h}.
\end{aligned}
\end{equation*}
By the triangle inequality, it yields
$$\|\bl u^E-\widetilde I_h \bl u^E\|_{0,\Omega_h}\leq \|\bl u^E- I_h \bl u^E\|_{0,\Omega_h}+\|I_h\bl u^E-\widetilde I_h \bl u^E\|_{0,\Omega_h}.$$
By (\ref{Pi_equal}), we obtain
\begin{equation*}
\begin{aligned}
\|I_h\bl u^E-\widetilde I_h \bl u^E\|_{0,\Omega_h}^2&=\sum_{K\in\T_h^b}\|I_K\bl u^E-\widetilde I_K \bl u^E\|_{0,K}^2\\
&=\sum_{K\in\T_h^b}\int_K \left|\sum_{i=0}^k\int_e \bl u^E\cdot \bl \n_h\,q_i\dif s\,\psi_i^e\right|^2 \dif x\\
&\lesssim \sum_{K\in\T_h^b} h\|\bl u^E\cdot \bl \n_h\|_{0,e}^2\sum_{i=0}^k\|\psi_i^e\|_{0,K}^2\\
&\lesssim \sum_{K\in\T_h^b} h\|\bl u^E\cdot \bl \n_h\|_{0,e}^2.
\end{aligned}
\end{equation*}
Notice that $(\bl u^E\cdot\bl\n)\circ M_h=0$, it follows from (\ref{deta}),(\ref{Omega_he})
and Lemmas \ref{Omega_he2K}, \ref{lem:extension} that
\begin{equation*}
\begin{aligned}
\sum_{e\in\E_h^b}\|\bl u^E\cdot \bl \n_h\|_{0,e}^2\lesssim &\sum_{e\in\E_h^b}(\|\bl u^E\cdot (\bl \n_h-\wt{\bl\n})\|_{0,e}^2
+\|(\bl u^E-\bl u^E\circ M_h)\cdot\wt{\bl\n}\|_{0,e}^2)\\
\lesssim&~h^2\sum_{e\in\E_h^b}\|\bl u^E\|_{0,e}^2+\delta\|\nabla \bl u^E\|_{0,(\Omega\setminus\Omega_h)\cup(\Omega_h\setminus\Omega)}^2\\
\lesssim&~h^2\sum_{e\in\E_h^b}\|\bl u^E\|_{0,\wt{e}}^2+\delta\|\nabla \bl u^E\|_{0,(\Omega\setminus\Omega_h)\cup(\Omega_h\setminus\Omega)}^2\\
\lesssim&~ h^2\|\bl u^E\|_{0,\Gamma}^2+h\delta\|\nabla \bl u^E\|_{0,\T_h^b}^2\\
\lesssim&~ h^2\|\bl u^E\|_{1,\Omega}^2+h^3\|\nabla \bl u^E\|_{0,\T_h^b}^2\\
\lesssim&~ h^2\|\bl u^E\|_{2,\Omega}^2,
\end{aligned}
\end{equation*}
where we have used the Sobolev trace theorem on $\Omega$.
Finally, by the Lemma \ref{lem:without_errapp} and the error estimate of $\|p^E-p_I\|_{0,\Omega_h}$ in Lemma \ref{lem:interpolation_app}, we complete this lemma.
\end{proof}

{From the above analysis, it is clear that the loss of accuracy without boundary value correction arises from using straight triangular elements to approximate curved elements on the boundary. 
This results in a loss of accuracy in the interpolation error $\|\bl u^E-\wt{I_h}\bl u^E\|_{0,\Omega_h}$.
Therefore, we consider a boundary value correction method to compensate for the geometric discrepancy between the curved boundary $\Gamma$ and the approximated boundary $\Gamma_h$.}

\section{Numerical experiment}\label{sec7:numerical}

In this section, we test three examples to validate error results of the original mixed element method (\ref{Without_Ph}), without any boundary correction and the correction method (\ref{Wh1}).
Consider Darcy problem on a circular domain $\{(x,y):x^2+y^2\leq 1\}$ {and a ring domain $\{(x,y):1/4\leq x^2+y^2\leq 1\}$} with triangular meshes. 
Before giving the numerical results, we define errors
$$\quad e_u=\|\bl u-\bl u_h\|_{0,\Omega_h},\qquad e_p=\|\nabla (p-p_h)\|_{0,\T_h}.$$

According to Theorems \ref{thm:without_err} and \ref{thm:L2err}, we expect to have a loss of accuracy for original mixed element method, 
a suboptimal $O(h^{k+1/2})$ convergence for correction method. 

\vspace{0.2cm}
\noindent{Example 1. Considering the problem (\ref{primal_problem}) with solution}
\vspace{0.2cm}
\begin{equation*}
\begin{aligned}
\bl u(x,y)=\dbinom{x^3(3x^2+2y^2-3)}{x^4y},\qquad p(x,y)=-\frac{1}{2}x^6-\frac{1}{2}x^4y^2+\frac{3}{4}x^4-\frac{3}{64},
\end{aligned}
\end{equation*}
which satisfies a homogeneous Neumann boundary condition $\bl u\cdot\bl\n=0$ on $\Gamma$ and $\int_{\Omega}p\dif x=0$.

\vspace{0.2cm}
\noindent Example 2. The exact solution is
\vspace{0.2cm}
\begin{equation*}
\begin{aligned}
\bl u(x,y)=\dbinom{2\pi \cos(2\pi x)\sin(2\pi y)}{2\pi\sin(2\pi x)\cos(2\pi y)},\qquad p(x,y)=-\sin(2\pi x)\sin(2\pi y),
\end{aligned}
\end{equation*}
which satisfies a non-homogeneous Neumann boundary condition $\bl u\cdot\bl\n\neq 0$ on $\Gamma$ and $\int_{\Omega}p\dif x=0$.

\vspace{0.2cm}
\noindent Example 3. The exact solution is
\vspace{0.2cm}
\begin{equation*}
\begin{aligned}
\bl u(x,y)=\dbinom{\pi \cos(\pi x)(e^y-e^{-y})}{\pi\sin(\pi x)(e^y+e^{-y})},\qquad p(x,y)=-(e^y-e^{-y})\sin(\pi x),
\end{aligned}
\end{equation*}
which satisfies a non-homogeneous Neumann boundary condition $\bl u\cdot\bl\n\neq 0$ on $\Gamma$ and $\int_{\Omega}p\dif x=0$.

As a comparison, we first solve the problem by standard Raviart-Thomas element method on two different domains, without any boundary value correction. 
Tabs. \ref{tab_eg1_Nu}-\ref{tab_eg3_Nu} show an $O(h^{1/2})$ convergence in $L^2$-norm for  velocity field when $k=1,2,3$, which is a loss of accuracy, as expected.

\begin{table}[!h]
\caption{\small\textit{Errors of Example 1 for velocity without boundary value correction on circle domain.}}
\label{tab_eg1_Nu}
\begin{center}
{\small
\begin{tabular}{c | c| c| c| c| c|c}
\hline
\multirow{2}{*}{$h$}&\multicolumn{2}{c|}{$k=1$}&\multicolumn{2}{c|}{$k=2$} &\multicolumn{2}{c}{$k=3$} \\
\cline{2-7}
          & $e_{u}$  & order  & $e_{u}$ & order & $e_{u}$  & order\\
\hline
1/8   &   7.26e-03   &   --    &   6.19e-03   &   --    &   6.14e-03   &   --   \\
\hline
1/16   &   2.13e-03   &   1.77   &   1.79e-03   &   1.79   &  1.76e-03   &   1.81  \\
\hline
1/32   &   6.11e-04   &   1.80   &   4.98e-04   &   1.84  & 4.83e-04   &   1.86   \\
\hline
1/64   &   1.84e-04   &   1.73   &   1.45e-04   &   1.78   & 1.38e-04   &   1.81   \\
\hline
\end{tabular}}
\end{center}
\end{table}

\begin{table}[!h]
\caption{\small\textit{Errors of Example 2 for velocity without boundary value correction  on circle domain.}}
\label{tab_eg2_Nu}
\begin{center}
{\small
\begin{tabular}{c | c| c| c| c| c|c}
\hline
\multirow{2}{*}{$h$}&\multicolumn{2}{c|}{$k=1$}&\multicolumn{2}{c|}{$k=2$} &\multicolumn{2}{c}{$k=3$} \\
\cline{2-7}
          & $e_{u}$  & order  & $e_{u}$ & order & $e_{u}$  & order\\
\hline
1/8   &   2.86e-02   &   --    &   1.46e-02   &   --     &   1.39e-02   &   --   \\
\hline
1/16   &   8.40e-03   &   1.77   &   4.88e-03   &   1.58   &  4.61e-03   &   1.59  \\
\hline
1/32   &   2.53e-03   &   1.73   &   1.61e-03   &   1.60  & 1.51e-03   &   1.61   \\
\hline
1/64   &   8.03e-04   &   1.66   &   5.45e-04   &   1.57   & 5.04e-04   &   1.58   \\
\hline
\end{tabular}}
\end{center}
\end{table}

\begin{table}[!h]
\caption{\small\textit{Errors of Example 3 for velocity without boundary value correction  on ring domain.}}
\label{tab_eg3_Nu}
\begin{center}
{\small
\begin{tabular}{c | c| c| c| c| c|c}
\hline
\multirow{2}{*}{$h$}&\multicolumn{2}{c|}{$k=1$}&\multicolumn{2}{c|}{$k=2$} &\multicolumn{2}{c}{$k=3$} \\
\cline{2-7}
          & $e_{u}$  & order  & $e_{u}$ & order & $e_{u}$  & order\\
\hline
1/8   &   6.57e-02   &   --    &   4.09e-02   &   --     &   3.64e-02   &   --   \\
\hline
1/16   &   2.17e-02   &   1.60   &   1.42e-02   &   1.52   &  1.26e-02   &   1.53  \\
\hline
1/32   &   7.49e-03   &   1.53   &   4.99e-03   &   1.51  & 4.42e-03   &   1.51   \\
\hline
1/64   &   2.59e-03   &   1.53   &   1.76e-03   &   1.51   & 1.54e-03   &   1.52  \\
\hline
\end{tabular}}
\end{center}
\end{table}

 We then present the results using the above tree examples to test the discrete scheme (\ref{Wh1}) on a circle domain and a ring domain. 
 Results of boundary value correction are shown in Tabs \ref{tab_eg1_u}-\ref{tab_eg3_p}.
 We observe an $O(h^{k+1/2})$ convergence in $L^2$-norm for velocity and an $O(h^k)$ convergence for pressure in $H^1$-norm, which agrees well with Theorem \ref{thm:L2err}.
 Tab. \ref{tab_eg3_u} shows that some examples may exhibit a superconvergence error for velocity in $L^2$-norm.

\begin{table}[!h]
\caption{\small\textit{Errors of Example 1 for velocity with boundary value correction  on circle domain.}}
\label{tab_eg1_u}
\begin{center}
{\small
\begin{tabular}{c | c| c| c|c|c|c}
\hline
\multirow{2}{*}{$h$}&\multicolumn{2}{c|}{$k=1$}&\multicolumn{2}{c|}{$k=2$} &\multicolumn{2}{c}{$k=3$} \\
\cline{2-7}
          & $e_u$  & order  & $e_u$ & order & $e_u$  & order\\
\hline
1/8   &     4.56e-03   &   --   &   2.17e-04   &   --   &   5.78e-06   &   --   \\
\hline
1/16   &   1.38e-03   &   1.72   &   3.57e-05   &   2.60   &   4.86e-07   &   3.57   \\
\hline
1/32   &    4.14e-04   &   1.74   &   5.62e-06   &   2.67  &   3.89e-08   &   3.64   \\
\hline
1/64   &   1.23e-04   &   1.75   &   8.68e-07   &   2.69    &   3.06e-09   &   3.67   \\
\hline
\end{tabular}}
\end{center}
\end{table}
\begin{table}[!h]
\caption{\small\textit{Errors of Example 1 for pressure with boundary value correction on circle domain.}}
\label{tab_eg1_p}
\begin{center}
{\small
\begin{tabular}{c | c| c| c| c| c|c}
\hline
\multirow{2}{*}{$h$}&\multicolumn{2}{c|}{$k=1$}&\multicolumn{2}{c|}{$k=2$} &\multicolumn{2}{c}{$k=3$} \\
\cline{2-7}
          & $e_p$  & order  & $e_p$  & order & $e_p$  & order\\
\hline
1/8    &   5.75e-02   &   --  &   7.14e-03   &   -- &   3.93e-04   &   -- \\
\hline
1/16   &   2.98e-02   &   0.95  &  1.86e-03   &   1.94    &   5.04e-05   &   2.96\\
\hline
1/32   &   1.49e-02   &   1.00  &  4.67e-04   &   2.00 &   6.21e-06   &   3.02 \\
\hline
1/64   &   7.38e-03   &   1.01   & 1.16e-04   &   2.00 &   7.71e-07   &   3.01  \\
\hline
\end{tabular}}
\end{center}
\end{table}
\begin{table}[!h]
\caption{\small\textit{Errors of Example 2 for velocity with boundary value correction on circle domain.}}
\label{tab_eg2_u}
\begin{center}
{\small
\begin{tabular}{c | c| c| c|c|c|c}
\hline
\multirow{2}{*}{$h$}&\multicolumn{2}{c|}{$k=1$}&\multicolumn{2}{c|}{$k=2$} &\multicolumn{2}{c}{$k=3$} \\
\cline{2-7}
          & $e_u$  & order  & $e_u$ & order & $e_u$  & order\\
\hline
1/8   &     2.52e-02   &   --   &   1.15e-03   &   --   &   3.35e-05   &   --   \\
\hline
1/16   &   7.37e-03   &   1.78   &   1.89e-04   &   2.60   &   2.81e-06   &   3.58   \\
\hline
1/32   &    2.14e-03   &   1.78   &  2.99e-05   &   2.66  &   2.22e-07   &   3.66   \\
\hline
1/64   &   6.22e-04   &   1.78   &   4.59e-06   &   2.70    &   1.72e-08   &   3.69   \\
\hline
\end{tabular}}
\end{center}
\end{table}
\begin{table}[!h]
\caption{\small\textit{Errors of Example 2 for pressure with boundary value correction on circle domain.}}
\label{tab_eg2_p}
\begin{center}
{\small
\begin{tabular}{c | c| c| c| c| c|c}
\hline
\multirow{2}{*}{$h$}&\multicolumn{2}{c|}{$k=1$}&\multicolumn{2}{c|}{$k=2$} &\multicolumn{2}{c}{$k=3$} \\
\cline{2-7}
          & $e_p$  & order  & $e_p$  & order & $e_p$  & order\\
\hline
1/8    &   4.85e-01   &   --  &   3.19e-02   &   -- &   1.57e-03   &   -- \\
\hline
1/16   &   2.43e-01   &   0.99  &  7.87e-03   &   2.02    &   1.95e-04   &   3.01\\
\hline
1/32   &   1.21e-01   &   1.01  &  1.92e-03   &   2.04 &   2.33e-05   &   3.06\\
\hline
1/64   &   5.89e-02   &   1.04    & 4.59e-04   &   2.06 &   2.69e-06   &   3.12  \\
\hline
\end{tabular}}
\end{center}
\end{table}

\begin{table}[!h]
\caption{\small\textit{Errors of Example 3 for velocity with boundary value correction on ring domain.}}
\label{tab_eg3_u}
\begin{center}
{\small
\begin{tabular}{c | c| c| c|c|c|c}
\hline
\multirow{2}{*}{$h$}&\multicolumn{2}{c|}{$k=1$}&\multicolumn{2}{c|}{$k=2$} &\multicolumn{2}{c}{$k=3$} \\
\cline{2-7}
          & $e_u$  & order  & $e_u$ & order & $e_u$  & order\\
\hline
1/8   &     4.77e-02   &   --   &   1.15e-03   &   --   &   1.50e-05   &   --   \\
\hline
1/16   &   1.55e-02   &   1.54   &   1.83e-04   &   2.65   &   8.55e-07   &   4.13   \\
\hline
1/32   &    5.25e-03   &   1.48   &  3.06e-05   &   2.58  &   5.02e-08   &   4.09   \\
\hline
1/64   &   1.81e-03   &   1.45   &   5.26e-06   &   2.54    &  2.92e-09   &  4.02   \\
\hline
\end{tabular}}
\end{center}
\end{table}
\begin{table}[!h]
\caption{\small\textit{Errors of Example 3 for pressure with boundary value correction on ring domain.}}
\label{tab_eg3_p}
\begin{center}
{\small
\begin{tabular}{c | c| c| c| c| c|c}
\hline
\multirow{2}{*}{$h$}&\multicolumn{2}{c|}{$k=1$}&\multicolumn{2}{c|}{$k=2$} &\multicolumn{2}{c}{$k=3$} \\
\cline{2-7}
          & $e_p$  & order  & $e_p$  & order & $e_p$  & order\\
\hline
1/8    &   4.33e-01   &   --  &   2.26e-02   &   -- &   7.78e-04   &   -- \\
\hline
1/16   &   2.18e-01   &   0.99  &  5.66e-03   &   2.00    &   9.81e-05   &   2.99\\
\hline
1/32   &   1.09e-01   &   1.00  &  1.42e-03   &   2.00 &   1.23e-05   &   3.00\\
\hline
1/64   &   5.44e-02   &   1.00    & 3.54e-04   &   2.00 &   1.54e-06  &   3.00  \\
\hline
\end{tabular}}
\end{center}
\end{table}

\section{Conclusion}\label{sec8:conclusion}
In this paper, we focus on the high-order Raviart-Thomas element on domains with curved boundaries.
The Neumann boundary is weakly imposed on the variational formulation.
This paper reaches an $O(h^{k+1/2})$ convergence in $L^2$-norm estimate
for the velocity field and an $O(h^k)$ convergence in $H^1$-norm estimate for the pressure.
Moreover, this paper concludes a completely theoretical analysis of a loss of approximation accuracy for high-order element.
 \paragraph*{Acknowledgment}
Yongli Hou and Yi Liu are supported by the NSFC grant 12171244.

\bibliographystyle{abbrv}

\begin{thebibliography}{35}
\bibitem{high_order_SBM}
N. M. Atallah, C. Canuto, and G. Scovazzi,
The high-order shifted boundary method and its analysis,
{Computer Methods in Applied Mechanics and Engineering}, 394:114885, 2022.


   \bibitem{Bertrand_2014_lowest}
   F. Bertrand, S. Munzenmaier, and G. Starke. First-order system least squares on curved
boundaries: Lowest-order Raviart-Thomas elements, SIAM Journal on Numerical Analysis, 52:880-894, 2014.
\bibitem{Bertrand_2014_High_order}
F. Bertrand, S. Munzenmaier, and G. Starke, First-order system least squares on curved
boundaries: Higher-order Raviart-Thomas elements, SIAM Journal on Numerical Analysis, 52:3165-3180, 2014.
\bibitem{Bertrand_2016_High_order}
F. Bertrand and G. Starke, Parametric Raviart-Thomas elements for mixed methods on domains with curved surfaces,
SIAM Journal on Numerical Analysis, 54:3648-3667, 2016.

\bibitem{book:FE_RT}
D. Braess,
Finite elements: theory, fast solvers, and applications in solid mechanics 3rd ed.
Springer-Verlag, 2012.

  \bibitem{1972}
 J. H. Bramble, T. Dupont, and V. Thom\'{e}e,
 Projection methods for Dirichlet's problem in approximating polygonal domains with boundary-value corrections,
 {Mathematics of Computation} 26:869-879, 1972.

 \bibitem{Bramble_King1994}
 J. H. Bramble and J. T. King,
 A robust finite element method for nonhomogeneous Dirichlet problems in domains with curved boundaries,
 {Mathematics of Computation} 63:1-17, 1994.


\bibitem{Poincare_Friedrichs_2003}
S. C. Brenner,
Poincar\'{e}-Friedrichs inequalities for piecewise $H^1$ functions,
SIAM Journal on Numerical Analysis, 41:306-324, 2003.

 \bibitem{FEM_Brenner}
S.~C. Brenner and L.~R. Scott,
{\em The mathematical theory of finite element methods (third
  edition)}, Springer-Verlag, 1994.


\bibitem{Burman_boundary_value_correction_2018}
E. Burman, P. Hansbo, and M. G. Larson,
A cut finite element method with boundary value correction,
{Mathematics of Computation} 87:633-657, 2018.

\bibitem{Puppi_mixed}
E. Burman and R. Puppi,
Two mixed finite element formulations for the weak imposition of the Neumann boundary conditions for the Darcy flow,
Journal of Numerical Mathematics, 30:141-162, 2022.


 \bibitem{Gunzburger_2019}
 J. Cheung, M. Perego, P. Bochev and M. Gunzburger,
 Optimally accurate higher-order finite element methods on polytopial approximations of domains with smooth boundaries,
 {Mathematics of Computation}, 88:2187-2219, 2019.

\bibitem{Cockburn_Boundary_conforming_DG_2009}
B. Cockburn, D. Gupta, and F. Reitich,
 Boundary-conforming discontinuous Galerkin methods via extensions from subdomains,
 {Journal of Scientific Computing}, 42:144-184, 2009.


\bibitem{Cockburn2012}
 B. Cockburn, M. Solano,
 Solving Dirichlet boundary-value problems on curved domains by extensions from subdomains,
{SIAM Journal on Scientific Computing} 34:A497-A519, 2012.

  \bibitem{Isogeometric_Analysis_Cottrell_2009}
 J. A. Cottrell, T. J. R. Hughes, and Y. Bazilevs,
 Isogeometric Analysis: Toward Integration of CAD and FEA,
 John Wiley $\&$ Sons, Ltd., Chichester, 2009.
%
\bibitem{fractual_non_matching_2012}
C. D'Angelo, A. Scotti, A mixed finite element method for Darcy flow in fractured porous media with non-matching grids,
ESAIM: Mathematical Modelling and Numerical Analysis 46: 465-489, 2012.

\bibitem{MFEM}
R. Dur\'{a}n, Mixed finite element methods, vol. 1939 of Lecture Notes in Mathematics, Springer-Verlag,
Berlin, 1-44, 2008, lectures given at the C.I.M.E. Summer School held in Cetraro, June 26-July 1,
2006. Edited by D. Boffi and L. Gastaldi.

 \bibitem{isoparametric_1968}
 I. Ergatoudis, B. Irons, and O. Zienkiewicz, Curved, isoparametric, quadrilateral elements for finite element analysis,
{International Journal of Solids and Structures}, 4:31-42, 1968.

\bibitem{book:RT_app}
 G. N. Gatica, A simple introduction to the mixed finite element method theory and applications, 
{Springer Briefs in Mathematics}, 2014.

\bibitem{Isogeometric_analysis_Hughes_2005}
T. J. R. Hughes, J. A. Cottrell, and Y. Bazilevs,
 Isogeometric analysis: CAD, finite elements, NURBS, exact geometry and mesh refinement,
 {Computer Methods in Applied Mechanics and Engineering}, 194:4135-4195, 2005.

\bibitem{Nitsche}
M. Juntunen, R. Stenberg,
 Nitsche's method for general boundary conditions,
 {Mathematics of Comoutation}, 78:1353-1374, 2009.

\bibitem{2023_div_perserve_unfitted_mixed_FEM}
 C. Lehrenfeld, T. V. Beeck. I. Voulis,
 Analysis of divergence-perserving unfitted finite element method for the mixed Poisson problem, arXiv:2306.12722.


\bibitem{Optimal_isoparametric_finite_elements_1986}
 M. Lenoir,
 Optimal isoparametric finite elements and error estimates for domains involving curved boundaries,
 {SIAM Journal on Numerical Analysis}, 23:562-580, 1986.


\bibitem{Yiliu_WG}
Y. Liu, W. Chen, Y. Wang,
A weak Galerkin mixed finite element method for second order elliptic equations on 2D curved domains,
{Communications in Computational Physics}, 32:1094-1128, 2022.

\bibitem{Part1}
A. Main, G. Scovazzi,
The shifted boundary method for embedded domain computations. Part I: Poisson and Stokes problems,
{Journal of Computational Physics}, 372:972-995, 2018.


\bibitem{Puppi_cutFEM}
R. Puppi,
A cut finite element method for the Darcy problem, Preprint arXiv:2111.09922.

\bibitem{book:Extension}
E. M. Stein, Singular integrals and differentiability properties of functions, 
Princeton University Press, 2, 1970.

\bibitem{Strang_Berger_The_change_in_solution_due_to_change_in_domain1973}
G. Strang and A. E. Berger,
The change in solution due to change in domain,
Partial Differential Equations (D. C. Spencer, ed.), Proc. Sympos. Pure Math., vol. 23, Amer. Math.
Soc, Providence, RI, 199-205, 1973.

\bibitem{Thomee_Polygonal_domain_approximation_in_Dirichlet's_problem1973}
V. Thom\'{e}e,
Polygonal domain approximation in Dirichlet's problem,
{Ima Journal of Applied Mathematics},
11:33-44, 1973.

\bibitem{Thomee_Approximate_solution_of_Dirichlet's_problem_using_approximating_polygonal_domains1973}
V. Thom\'{e}e,
Approximate solution of Dirichlet's problem using approximating polygonal domains,
Topics in Numerical Analysis (J. J. H. Miller, ed.), Academic Press, New York, 311-328, 1973.

\end{thebibliography}

\end{document}